\documentclass[12pt,a4paper,oneside]{amsart}
\usepackage[utf8]{inputenc}
\usepackage{amsmath,amscd,hyperref, amsfonts, amssymb, amsthm}
\usepackage{setspace,kantlipsum}
\usepackage{textcomp}
\usepackage[margin=1in]{geometry}
\usepackage[balance,small,nohug]{diagrams}
\usepackage{tikz}
\usetikzlibrary{arrows.meta}
\theoremstyle{plain} 
\newtheorem{thm}{Theorem}[section]
\newtheorem{lem}[thm]{Lemma}
\newtheorem{prop}[thm]{Proposition}
\newtheorem{corollary}[thm]{Corollary}

\theoremstyle{definition}

\newtheorem{definition}[thm]{Definition}
\newtheorem{example}[thm]{Example}
\newtheorem{remark}[thm]{Remark}

\newtheorem{problem}[thm]{Problem}

\newcounter{tmp}

\begin{document}

\title{Nef cones of nested Hilbert schemes of points on surfaces}

\author{Tim Ryan}
\thanks{During the preparation of this article, the first author was partially supported by the NSF grant DMS-1547145.}
\address{Simons Center for Geometry and Physics/Stony Brook University, Stony Brook, NY 11794}
\email{timothy.ryan@stonybrook.edu}

\author{Ruijie Yang}
\thanks{Submitted August 2, 2017.}
\address{Stony Brook University, Stony Brook, NY 11794}
\email{ruijie.yang@stonybrook.edu}

\begin{abstract}
    Let $X$ be the projective plane, a Hirzebruch surface, or a general $K3$ surface. In this paper, we study the birational geometry of various nested Hilbert schemes of points parameterizing pairs of zero-dimensional subschemes on $X$. We calculate the nef cone for two types of nested Hilbert schemes. As an application, we recover a theorem of Butler about syzygies on Hirzebruch surfaces. 

\end{abstract}


\maketitle

\section{Introduction}
In this paper, we will show that many of the methods developed to study the birational geometry of Hilbert schemes of points on surfaces apply to the study of nested Hilbert schemes of points as well. 
In particular, we will compute the nef cones of two types of nested Hilbert schemes for rational surfaces and general $K3$ surfaces.
These will provide examples of nef cones for varieties with high dimension and high Picard rank including some that are not $\mathbf{Q}$-factorial.

Let $X$ be a smooth projective surface over $\mathbf{C}$ and let $X^{[n]}$ denote the Hilbert scheme parametrizing zero dimensional subschemes of length $n$ of $X$. 
The Hilbert scheme $X^{[n]}$ is a smooth projective variety of dimension $2n$ by \cite{Fogarty2}. 
In recent years, the birational geometry of $X^{[n]}$ has been extensively studied.
In particular, the minimal model program (MMP) has been run on it for many specific types of surface $X$ (e.g. for $\mathbf{P}^2$ in \cite{ABCH} and \cite{LZ}).

On the other hand, there are various natural generalizations of $X^{[n]}$ about which far less is known, in particular nested Hilbert schemes of points. 
They appear naturally in several contexts. 
First, they appear in the study of topology of Hilbert schemes of points. For example, they were used to define the representation of the Heisenberg algebra on $\oplus_n H^{\ast}(X^{[n]},\mathbf{Q})$ in \cite{Nakajima2}.
Secondly, they appear in the study of knot invariants (e.g. \cite{GORS, OS}). 
Finally, they appear in the study of commuting varieties of parabolic subalgebras of $\mathrm{GL}(n)$ (c.f \cite{GG, BB}).

Here we mainly consider
\[ X^{[n+1,n]} \subset X^{[n+1]} \times X^{[n]} \] 
which parametrizes pairs of subschemes $(z,z')$ of length $n+1$ and $n$, respectively, on $X$ such that $z'$ is a subscheme of $z$. 
It's known that $X^{[n+1,n]}$ is a smooth variety of dimension $2n+2$ by \cite[Chapter 1]{Cheah}. 
It is equipped with two projection maps, 
\[ \mathrm{pr}_a:X^{[n+1,n]}\to X^{[n+1]} \text{ and } \mathrm{pr}_b:X^{[n+1,n]}\to X^{[n]}. \]
Near the end of the paper, we will also study the universal family $Z^{[n]} \subset X^{[n]} \times X$ which is the nested Hilbert scheme $X^{[n,1]}$.

The eventual goal would be to run the MMP on these spaces. Often the first step of running the MMP is computing the nef cone.
Recall that a Cartier divisor $D$ on an irreducible projective variety $Y$ is \textit{nef} if $c_1(\mathcal{O}(D))\cdot C\geq 0$ for every irreducible curve $C\subset Y$, where $\mathcal{O}(D)$ is the associated line bundle of $D$ on $Y$. 
This notion naturally extends to $\mathbf{R}$-Cartier divisors (see \cite[Section 1.4]{Lazarsfeld1}). The \textit{nef cone} of $Y$ is the convex cone of the numerical classes of all nef $\mathbf{R}$-Cartier divisors on $Y$. 

Using classical methods (i.e. $k$-very ample line bundles \cite{CG}), the nef cone was computed for Hilbert schemes of points on the projective plane \cite{LQZ}, on the product of projective lines, on Hirzebruch surfaces, and on del Pezzo surfaces of degree at least two \cite{BC}. 
Recently, there has been tremendous progress using Bridgeland stability to compute the nef cones of Hilbert schemes of points on K3 surfaces in 
\cite{BM2}, on Abelian surfaces (c.f. \cite{MM,YY}), on Enriques surfaces (c.f. \cite{Nue}), and on all surfaces with Picard number one and irregularity zero 
\cite{BHLRSWZ}.

To calculate the nef cone for nested Hilbert schemes, we must first understand the Picard group and the N\'eron-Severi group. 
\begingroup
\setcounter{tmp}{\value{thm}}
\setcounter{thm}{0} 
\renewcommand\thethm{\Alph{thm}}
\begin{prop}
Let $X$ be a irreducible smooth projective surface of irregularity zero and fix $n\geq 2$. Then
\[\mathrm{Pic}\left(X^{[n+1,n]} \right) \cong \mathrm{Pic}(X)^{\oplus 2} \oplus \mathbf{Z}^2.\]
In particular, the N\'eron-Severi group $N^1(X^{[n+1,n]})$ has rank $2(\rho(X)+1)$, where $\rho(X)$ is the Picard number of $X$.
\end{prop}
\endgroup

Knowing the Picard groups, we can describe the nef cones. 
To easily state our theorem, let us first recall the nef cone of the Hilbert schemes of points on $\mathbf{P}^2$.
The nef cone of $\left(\mathbf{P}^{2}\right)^{[n]}$ is spanned by the two divisors 
\[ H[n] \text{ and } D_{n-1}[n] = (n-1)H[n] -\frac{1}{2}B[n], \]
where $H[n]$ is the class of the pull-back of the ample generator via the Hilbert-Chow morphism and $B[n]$ is the exceptional locus.

\begingroup
\setcounter{tmp}{\value{thm}}
\setcounter{thm}{1}
\renewcommand\thethm{\Alph{thm}}
\begin{thm}
The nef cone of $\left(\mathbf{P}^{2}\right)^{[n+1,n]}$, $n>1$, is spanned by the four divisors \[ \mathrm{pr}_b^\ast(H[n]), \mathrm{pr}_b^\ast(D_{n-1}[n]), \mathrm{pr}_a^\ast(H[n+1])-\mathrm{pr}_b^\ast(H[n]), \text{ and }\mathrm{pr}_a^\ast(D_{n}[n+1]). \]
Similar results hold for the Hirzebruch surfaces $\mathbf{F}_i$, $i\geq 0$, and general $K3$ surfaces $S$ as well as for the universal families on all of these surfaces.
\end{thm}
\endgroup

Knowing the nef cone allows us to recover a special case of \cite[Theorem 2B]{Butler} about the projective normality of line bundles on Hirzebruch surfaces.
\begingroup
\setcounter{tmp}{\value{thm}}
\setcounter{thm}{2}
\renewcommand\thethm{\Alph{thm}}
\begin{prop}
Let $X$ be a Hirzebruch surface, and $A$ be an ample line bundles on $X$, then $L=K_X+nA$ is projectively normal for $n\geq 4$.
\end{prop}
\endgroup

Finally, the computation of the nef cone is only one step in understanding the geometry of these spaces so we conclude with a list of open and related problems.

\subsection{Organization}
In Section \ref{sec: prelim}, we review basic facts about (nested) Hilbert schemes.
In Section \ref{sec: pic}, we show that for surfaces of irregularity zero, the Picard group of $X^{[n+1,n]}$ is generated by $\mathrm{Pic}(X)^{\oplus 2}$ and two additional classes corresponding to divisors of nonreduced schemes.
In Section \ref{sec: div and curves}, we give a basis for the divisors and curves on $X^{[n+1,n]}$, and we define additional divisors and curves that will be used.
In Section \ref{sec: nef}, we compute the nef cones of $X^{[n+1,n]}$ where $X$ is the projective plane, a Hirzebruch surface, or a general K3 surface. 
In Section \ref{sec: univ fam}, we briefly work out the universal family case.
In Section \ref{sec: application}, we apply the nef cone computation to prove Proposition C.
Finally, in Section \ref{sec: problems}, we pose some open questions.

\section{Preliminaries}
\label{sec: prelim}
In this section, we will recall the basic definitions and properties of (nested) Hilbert schemes of points on surfaces.
Some standard references for Hilbert schemes of points are \cite{Fogarty1}, \cite{Nakajima} and \cite{Gottsche1}.
References for the more general setting of nested Hilbert schemes are \cite{Kleppe} and \cite{Cheah}.

Throughout the paper, let $X$ be a smooth irreducible projective surface over $\mathbf{C}$ of irregularity $q:=H^1(\mathcal{O}_X)=0$.


\subsection{Hilbert schemes of points}
Given a surface $X$, the symmetric group acts on its product $X^n$ by permuting the factors.
Taking the quotient by this action gives the \textit{symmetric product} of $X$, denoted $X^{(n)}$.
This space parametrizes unordered collections of $n$ points of $X$.

A natural desingularization of the symmetric product is the Hilbert scheme of $n$ points on $X$, denoted by $X^{[n]}$, which parametrizes subschemes of $X$ with Hilbert polynomial $n$.
It resolves $X^{(n)}$ via the \textit{Hilbert-Chow morphism} 
\[\mathrm{hc}: X^{[n]} \to X^{(n)},\]
which maps a subscheme to its support with multiplicity.
It is well known that $X^{[n]}$ is an irreducible smooth projective variety of dimension $2n$ by \cite[Lemma 4.4]{Fogarty2}.

Using the Hilbert-Chow morphism, we want to associate each line bundle on $X$ to a line bundle on $X^{[n]}$.
Informally, we pull the line bundle back along each projection of $X^n$ onto its factors, tensor all of them so that it descends to the symmetric product, and then pull it back along the Hilbert-Chow morphism.

\begin{definition}
Let $L[n]$ be the class of the pull-back of $L^{(n)}$ by the Hilbert-Chow morphism $X^{[n]} \to X^{(n)}$, where $L^{(n)}$ is the descent of the $S_n$-equivariant line bundle $L^{\boxtimes n} \cong p_1^*(L) \otimes \cdots \otimes p_n^*(L)$ on $X^n$.
\end{definition}

Corollary 6.2 of \cite{Fogarty2} shows that 
\begin{eqnarray}\label{eqn:fogartypic}
\mathrm{Pic}\left(X^{[n]}\right) \cong \mathrm{Pic}(X) \oplus \mathbf{Z}\left(\frac{B}{2}\right),
\end{eqnarray}
where $\mathrm{Pic}\left(X\right) \subset \mathrm{Pic}\left(X^{[n]}\right)$ is embedded by $L \to L[n]$ and B is the locus of non-reduced schemes when $X$ has no irregularity.

We want to define another way to associate a line bundle on $X$ to one on $X^{[n]}$ via the universal family.
Let $Z^{[n]}\subset X^{[n]}\times X$ be the universal family over the Hilbert scheme which can be defined set-theoretically as
\[Z^{[n]} = \{\left(\xi,P \right): P \in \text{supp}(\xi) \}.\]
As a subset of the product, it has two natural projections (shown below).
\[\begin{diagram}
&& Z^{[n]} && \\
& \ldTo^{\mathrm{pr}_a} & & \rdTo^{\mathrm{pr}_b} &\\
X^{[n]} &&&& X
\end{diagram}\]

Using these, we pull the bundle back to the universal family, push it down onto the Hilbert scheme, and take the determinant (c.f. \cite{LP} \& \cite{CHW}). 

\begin{definition}\label{def:D_L}
Let $D_L[n]=\mathrm{det}\left(L^{[n]}\right)$, where $L^{[n]}:=\left(\mathrm{pr}_a \right)_{\ast} \mathrm{pr}_b^{\ast} L$ is the tautological bundle of $L$.
\end{definition}

By the Grothendieck-Riemann-Roch Theorem, it is not hard to see that $D_L[n]=L[n]-\frac{1}{2}B[n]$. 


\subsection{Nested Hilbert Schemes}

Hilbert schemes of points were generalized to nested Hilbert schemes of points in \cite{Kleppe}.
A nested Hilbert scheme of points parametrizes collections of subschemes $\left(z_m, \cdots, z_1\right)$ of finite length that are nested inside of each other $\left(\text{i.e. }z_m \supset  \cdots \supset z_1 \text{ as subschemes} \right)$.

The simplist nested Hilbert schemes of points parametrize collections of only two subschemes.
We will study two families of this type.

\subsubsection{Universal families}
One of the families is when the smaller subscheme has length one, which we generally will denote as $X^{[n,1]}$.
This nested Hilbert scheme is the universal family defined above, and it has been extensively studied so we will sometimes use the classical notation $Z^{[n]}$.
\begin{prop}[{\cite[Proposition 7.1]{Fogarty2}, \cite[Theorem 1.1]{Song}}]\label{prop:universal family}
The universal family $Z^{[n]}$ is a normal, irreducible, Cohen-Macaulay, singular ($n>2$), and non $\mathbf{Q}$-factorial variety of dimension $2n$. 
\end{prop}
Fogarty computed the Picard group of the universal family for any smooth projective surface.
We state the result under our additional assumption of irregularity zero.
\begin{prop}[{\cite[Theorem 7.6]{Fogarty2}}]
$\mathrm{Pic}\left(X^{[n,1]}\right)$ is spanned by the pull backs from the two natural projections 
so that
\[ \mathrm{Pic}\left(X^{[n,1]}\right) \cong \mathrm{Pic} \left(X^{[n]}\right) \oplus \mathrm{Pic}(X) \cong \mathrm{Pic}(X)^{\oplus 2} \oplus \mathbf{Z}.\]
\end{prop}

\subsubsection{Smooth nested Hilbert schemes}
The other family which we will study is when the lengths of the two subschemes differ by one.
This nested Hilbert scheme can be defined set-theoretically as
\[X^{[n+1,n]} = \{ (z,z') \in X^{[n+1]} \times X^{[n]} \ | \ z' \text{ is a closed subscheme of } z \}\]
and is a smooth variety of dimension  $2n+2$ by \cite[Chapter 1]{Cheah}.
These are a natural choice to study as they are the only smooth nested Hilbert schemes of points on a surface other than the Hilbert schemes themselves also by \cite{Cheah}. 


Similar to the universal family, the product $X^{[n+1]} \times X^{[n]}$ equips $X^{[n+1,n]}$ with two surjective projection maps. 
\[\begin{diagram}
&& X^{[n+1,n]} && \\
& \ldTo^{\mathrm{pr}_a} & & \rdTo^{\mathrm{pr}_b} &\\
X^{[n+1]} &&&& X^{[n]}
\end{diagram}\]
We will use these to compute the Picard groups of these nested Hillbert schemes.

\subsubsection{Residue maps}

The residue map 
\[ \mathrm{res}: X^{[n+1,n]} \to X \] is defined set-theoretically sending $(z,z')$ to $P$ where $P$ is the unique point $z$ and $z'$ differ.
It is a morphism by \cite[Section 2]{ES}.




\section{Picard group of the smooth nested Hilbert schemes}
\label{sec: pic}

In this section, we will compute the Picard group of $X^{[n+1,n]}$ for a smooth irreducible projective surface $X$ under our standing assumption of irregularity zero.

\begin{prop}\label{prop:Picard group}
$\mathrm{Pic}(X^{[n+1,n]})$ is generated by the pull backs of $\mathrm{Pic}(X^{[n+1]})$ and $\mathrm{Pic}(X^{[n]})$ along $\mathrm{pr}_a$ and $\mathrm{pr}_b$ respectively so
\[ \mathrm{Pic}(X^{[n+1,n]}) \cong \mathrm{Pic}(X)^{2}\oplus \mathbf{Z}^{2}. \]
As $\mathrm{Pic}^0(X)=0$, $\mathrm{Pic}(X^{[n+1,n]})$ has no continuous part so
\[ \rho(X^{[n+1,n]}) = 2(\rho(X)+1). \] 
\end{prop}

The method of proof is classical, but for the benefit of non-experts we give some details.
We calculate the Picard group of each nested Hilbert scheme on an open set that is large enough to not lose information about Picard group and that is small enough to not contain any really troublesome points (similar to Fogarty's approach \cite[Section 7]{Fogarty2}).
Specifically, we show that an open set of the nested Hilbert scheme is the blow up along a smooth subvariety of an open subset of a smooth variety; since we know the Picard group of that variety, we subsequently know the Picard group of the nested Hilbert scheme.

\subsection{Stratification of the universal family and blow up structure of nested Hilbert schemes}

Let's recall the construction of the stratification of $X^{[n]}\times X$ in \cite[Section 3]{ES}.
\begin{definition}
Let $W_{i,n}$ be the set of points $(\xi,P)\in X^{[n]}\times X$ such that the ideal $I_{Z^{[n]}}$ needs exactly $i$ generators at $\left(\xi,P\right)$. Equivalently,
\[ W_{i,n}= \{ (\xi,P) \in X^{[n]}\times X : \mathrm{dim}_{\mathbb{C}} \left( I_{\xi}(P) \right)=i \}. \]
where $I_{\xi}(P)=I_{\xi}\otimes \kappa(P)$ and $\kappa(P)$ is the residue field of $P$.
\end{definition}

\begin{remark}
To give this definition, we implicitly use the canonical isomorphism:
\[ I_{Z^{[n]},(\xi,P)} \cong I_{\xi,P}. \]
This can be obtained by restricting the ideal sheaf sequence of $Z^{[n]}$ inside $X^{[n]}\times X$ to $\{\xi\} \times X$:
\[ 0 \to I_{Z^{[n]}}|_{\xi \times X}\cong I_{\xi} \to \mathcal{O}_{X} \to \mathcal{O}_{\xi}\to 0.\]
Hence by further restricting to $\{\xi\} \times \{P\}$, one gets the desired isomorphism.
\end{remark}

It's easy to see that $W_{1,n}=(X^{[n]}\times X)\setminus Z^{[n]}$. By \cite[Proposition 7.1]{Fogarty2}, $W_{2,n}$ is the smooth locus of $Z^{[n]}$, which we denote by $Z^{[n]}_{\mathrm{sm}}$. For $i\geq 3$, it is known by the proof of \cite[Proposition 2.2]{ES} that
\begin{eqnarray}\label{eqn:codimension estimate}
\mathrm{codim}(W_{i,n},X^{[n]}\times X) \geq 2i-2. 
\end{eqnarray}

\begin{definition}
Denote $Z^{[n]}_{\mathrm{sing}}:= Z^{[n]} \backslash Z^{[n]}_{\mathrm{sm}}$ and $U:=(X^{[n]}\times X)\setminus Z^{[n]}_{\mathrm{sing}}=W_{1,n}\cup W_{2,n}$.
\end{definition}

Consider the map $\phi=\mathrm{pr}_b \times \mathrm{res}:X^{[n+1,n]} \to X^{[n]}\times X$. We want to investigate the fiber dimension over $W_{i,n}$. Let $(\eta,\xi) \in X^{[n+1,n]}$ and $P=\mathrm{res}(\eta,\xi) \in X$. 
Then there is a natural short exact sequence
\[ 0 \to I_{\eta}\to I_{\xi} \to \kappa(P) \to 0\]
Hence, $I_{\eta}(P)$ is a hypersurface inside $I_{\xi}(P)$ consisting of functions vanishing at $P$, and $I_{\eta}(Q)\cong I_{\xi}(Q)$ for any other $Q\in X$. 
Therefore the fiber $\phi^{-1}(\xi,P)$ is naturally identified with the projective space $\mathbf{P}(I_{\xi}(P))$. 
This implies that the fiber dimension over $W_{i,n}$ is $i-1$.
In the proof, we will use this to show that $\phi^{-1}(U)$ is a large open set of the nested Hilbert scheme.

Using the description of fibers, Ellingsurd and Str{\o}mme proved the following:
\begin{prop}({\cite[Proposition 2.2]{ES}})\label{prop:blow up of universal family}
The map $\phi:X^{[n+1,n]} \to X^{[n]}\times X$ is canonically isomorphic to the blowing up of $X^{[n]}\times X$ along $Z^{[n]}$ and the exceptional divisor $E$ corresponds to the divisor consists of $(\eta,\xi)$ where $\eta$ and $\xi$ have the same support. In particular, over $W_{2,n}=Z^{[n]}_{\mathrm{sm}}$, the morphism $\phi$ is a $\mathbf{P}^1$-bundle.
\end{prop}

Since blowing up commutes with flat base change and $U\cap Z^{[n]}=W_{2,n}\cap Z^{[n]}=Z^{[n]}_{\mathrm{sm}}$, we have the following lemma.
\begin{lem}\label{lem:blow up smooth locus}
$\phi^{-1}(U)$ is canonically isomorphic to $ \mathrm{Bl}_{Z^{[n]}_{\mathrm{sm}}}(U)$ via the restriction of $\phi$ to $\phi^{-1}(U)$.
\end{lem}


\subsection{Proof of the proposition}
Now we have all of the ingredients needed, so we can start the proof of Proposition \ref{prop:Picard group}. In the proof, an open subset of a variety whose complement has codimension at least two is called \textit{a large open subset}.
\begin{proof}
First, we show that $U$ and $\phi^{-1}(U)$ are large open sets in $X^{[n]}\times X$ and $X^{[n+1,n]}$ respectively. Since $U=W_{1,n}\cup W_{2,n}$, by equation (\ref{eqn:codimension estimate}), 
\[ \mathrm{codim}(X^{[n]}\times X\setminus U,X^{[n]}\times X) \geq \mathrm{min}_{i\geq 3}\{2i-2\}\geq 2.\] 
Therefore, $U$ is a large open subset of $X^{[n]}\times X$. 

Denote $X^{[n+1,n]}_{\ast}=\phi^{-1}(U)$. 
As the fiber dimension of $\phi$ over $W_{i,n}$ is $(i-1)$, 
\[ \mathrm{codim}(\phi^{-1}(W_{i,n}),X^{[n+1,n]}) \geq \mathrm{min}_{i\geq 3}\{(2i-2)-(i-1)\}\geq 2\]
for $i \geq 3$.
Hence, $X^{[n+1,n]}_{\ast}$ is also a large open subset of $X^{[n+1,n]}$.

Since $X^{[n+1,n]}$ and $X^{n}\times X$ are smooth, by \cite[Proposition 2.6.5]{Hart}, remove a closed subset of codimension at least two does not affect the Picard group. Therefore we have
\begin{eqnarray*}
\mathrm{Pic} \left( X^{[n]} \times X\right)\cong \mathrm{Pic}(U), \\
\mathrm{Pic} \left( X^{[n+1,n]}\right) \cong \mathrm{Pic} \left( X_\ast^{[n+1,n]}\right).
\end{eqnarray*}

Note that, by Lemma \ref{lem:blow up smooth locus}, $X^{[n+1,n]}_{\ast}$ is isomorphic to the blow up of $U$ along the smooth locus of the universal family.
Therefore, as $U$ is smooth and $Z^{[n]}_{\mathrm{sm}}$ is a smooth subvariety of $U$,
\[\mathrm{Pic}(X^{[n+1,n]}_{\ast})\cong \mathrm{Pic}(U)\oplus \mathbf{Z}(E_U)\]
where $E_U$ is the exceptional locus of $\phi \vert_U$.
Denote E to be the exceptional divisor of the blow up morphism $\phi$. 
Then, $E|_U=E_U$.

Finally, we recall that if $Y$ and $Z$ are smooth projective varieties, then 
\[\mathrm{Pic}(Y\times Z)\cong \mathrm{Pic}(Y) \oplus \mathrm{Pic}(Z) \oplus \mathrm{Hom}(A(Y),P(Z)),\]
where $A(Y)$ is the Albanese variety of $Y$ and $P(Z)$ is the Picard variety of $Z$ (\cite{Mumford}). 
We have $A(X)=0$ since we assumed that $H^1(X,\mathcal{O}_X)=0$.
Hence,
\[\mathrm{Pic}(X^{[n]}\times X)\cong \mathrm{Pic}(X^{[n]})\oplus \mathrm{Pic}(X).\]

Combining this with Fogarty's previously referenced result (\ref{eqn:fogartypic}), we obtain that
\begin{eqnarray*}
    \mathrm{Pic}(X^{[n+1,n]}) &\cong& \mathrm{Pic}(X^{[n+1,n]}_{\ast}) \cong \mathrm{Pic}(U)\oplus \mathbf{Z}(E_U)\\
                            &\cong& \mathrm{Pic}(X^{[n]}\times X)\oplus \mathbf{Z}(E_U) \cong \mathrm{Pic}(X^{[n]})\oplus \mathrm{Pic}(X)\oplus \mathbf{Z}(E_U)\\
                            &\cong& \mathrm{Pic}(X)^{\oplus 2} \oplus \mathbf{Z}(\frac{B[n]}{2}) \oplus \mathbf{Z}(E_U)
\end{eqnarray*}
Therefore,
\[ \mathrm{Pic}(X^{[n+1,n]})\cong \mathrm{Pic}(X)^{\oplus 2} \oplus \mathbf{Z}(\mathrm{pr}_b^\ast(\frac{B[n]}{2})) \oplus \mathbf{Z}(E)\]

Note that by Proposition \ref{prop:blow up of universal family}, $E$ corresponds to the locus of pairs $(\eta,\xi) \in X^{[n+1,n]}$ where $\eta$ and $\xi$ have same support. Hence 
$E =\mathrm{pr}_a^\ast \left(\frac{B[n+1]}{2}\right)-\mathrm{pr}_b^\ast\left(\frac{B[n]}{2}\right)$ (\cite[Section 1]{Lehn}), so the Picard group is generated by the pull backs the Picard groups analogously to the universal family case.
\end{proof}

\begin{remark}
One could in principle compute the Picard group when the irregularity is nonzero by including the input of divisorial correspondences (as in \cite[Proposition 7.1]{Fogarty2}).
Therefore, for simplicity, we only state the proposition under the assumption of zero irregularity.
\end{remark}

\section{Divisors and curves}
\label{sec: div and curves}

In this section, we introduce divisor and curve classes on $X^{[n+1,n]}$ that play a crucial role in its birational geometry.
The classes we define are generalizations of classes on $X^{[n]}$ so we will first recall those constructions.


\subsection{Divisors and Curves on the Hilbert Scheme}
Good references for these classes and their intersection product on rational surfaces are \cite{ABCH} and \cite{BC}.
Assuming the Picard number of $X$ is $k$, let $H_1$, $\cdots$, $H_k$ be a set of effective divisors on $X$ whose numerical classes span the N\'eron-Severi space $N^1_{\mathbb{R}}(X)$ and let $h_i$ be a general irreducible representative of the class $H_i$ which we now fix.


\subsubsection{Divisors on the Hilbert Scheme}
Recall that Fogarty \cite[Lemma 4.4]{Fogarty2} showed 
\[\mathrm{Pic}(X^{[n]})\cong \mathrm{Pic}(X)\oplus \mathbf{Z} \]
A classical $\mathbb{R}$-basis for the N\'eron-Severi space of $X^{[n]}$ is given by looking at the schemes whose support intersects those given representatives,
\[H_i[n] :=\{\xi\in X^{[n]} : \text{supp}(\xi) \cap h_i \neq \emptyset \},\]
along with the divisor of nonreduced schemes, which is defined set theoretically
\[B[n] := \{\xi \in X^{[n]} : \xi \text{ is nonreduced}\}.\]
$B[n]$ is the exceptional locus of the Hilbert-Chow morphism. 
Note, $H_i[n]$ is the same class as $H_i[n]$ defined in Section \ref{sec: prelim}.


\subsubsection{Curves on the Hilbert Scheme}
In contrast to the divisors, we will only define the types of curve class relevant for our proof, not an entire basis.
We will represent these curve classes with diagrams where hollow points represent a point of $\xi$ moving on the curve it lies on, solid points are fixed points in $\xi$, and dotted arrows attached to a point represent varying the tangent direction (i.e. double scheme structure) of a length two scheme on that point (c.f. \cite{Stathis}).
A basis for the space of curves on $X^{[n]}$ has a curve corresponding to each $H_i$ and one additional element corresponding to $B$.
For the first type of curve, fix $n-1$ general points and then vary the $n$-th point of the subscheme along the curve $h_i$; denote that curve by $C_i[n]$.
\begin{center}

\begin{tikzpicture}
\draw plot [smooth, tension=1] coordinates { (-.2,-1) (0,-.6) (0,.6) (.2,1)};
\node[draw=black,circle,scale=.6] at (0,0) {};
\node[fill,circle,scale=.5] at (1.2,.8) {};
\node[fill,circle,scale=.5] at (1.2,-.8) {};
\node[fill,circle,scale=.2] at (1.2,.3) {};
\node[fill,circle,scale=.2] at (1.2,0) {};
\node[fill,circle,scale=.2] at (1.2,-.3) {};
\node[right] at (1.7,0) {$n-1$ general fixed points};
\node[right] at (0,-1) {$h_i$};
\node[left] at (-1,0) {$C_i[n]   :$};
\end{tikzpicture}

\end{center}

For the last element of the basis, there are two natural choices which each have their own advantages.
The first choice is a curve where we fix $n-1$ general points and then vary a double scheme structure supported on a given one of the fixed points; denote this curve by $A[n]$.
\begin{center}
\begin{tikzpicture}
\node[fill,circle,scale=.5] at (0,0) {};
\draw[-{>[scale=1]},dashed,thick] (0,0) -- (1/2,1/2);
\node[fill,circle,scale=.5] at (1.4,.8) {};
\node[fill,circle,scale=.5] at (1.4,-.8) {};
\node[fill,circle,scale=.2] at (1.4,.3) {};
\node[fill,circle,scale=.2] at (1.4,0) {};
\node[fill,circle,scale=.2] at (1.4,-.3) {};
\node[right] at (1.9,0) {$n-2$ general fixed points};
\node[left] at (-1,0) {$A[n]:$};
\end{tikzpicture}
\end{center}
This curve class has the advantage of intersecting all of the $H_i[n]$ in zero but the disadvantage of being hard to use in some computations.

The alternative is to fix $n-2$ general points, fix a general point of $h_1$, and then vary the $n$-th point along $h_1$; denote this curve by $C_0[n]$.
\begin{center}
\begin{tikzpicture}[scale=1]
\draw plot [smooth, tension=1] coordinates { (-.2,-1) (0,-.6) (0,.6) (.2,1)};
\node[draw=black,circle,scale=.5] at (0,-.25) {};
\node[fill,circle,scale=.5] at (0,.25) {};
\node[fill,circle,scale=.5] at (1.2,.8) {};
\node[fill,circle,scale=.5] at (1.2,-.8) {};
\node[fill,circle,scale=.2] at (1.2,.3) {};
\node[fill,circle,scale=.2] at (1.2,0) {};
\node[fill,circle,scale=.2] at (1.2,-.3) {};
\node[right] at (1.7,0) {$n-2$ general fixed points};
\node[right] at (0,-1) {$h_1$};
\node[left] at (-1,0) {$C_0[n]:$};
\end{tikzpicture}
\end{center}
This curve has the advantage of being easy to use during computations.  

When the $n$ is clear, we will drop the $[n]$ from the notation for both curves and divisors.

\subsubsection{Other curves and divisors}\label{section:other curves}
We can also define a broader range of curve and divisor classes which are useful when computing nef and effective cones.

First, recall that if some multiple of the divisor can be written in the basis as $\sum_{i=1}^k m_i\cdot H_i-\frac{1}{2}B$, we denote this by $D_{\mathbf{m}}$.

Next, let $C_{\gamma,r}[n]$ be the curve class of $ \xi \subset X^{[n]}$ given by fixing $r-1$ general points of a curve $c$ of class $\gamma$ as part of $\xi$, $n-r$ general points not contained in $c$ as part of $\xi$, and varying the last point of $\xi$ on $c$.

\begin{center}
\begin{tikzpicture}[scale=1]
\draw plot [smooth, tension=.5] coordinates { (-.2,-1.5) (0,-1.1) (0,.6) (.2,1)};
\node[draw=black,circle,scale=.5] at (0,-1.1) {};
\node[fill,circle,scale=.5] at (0,.6) {};
\node[fill,circle,scale=.5] at (0,-.6) {};
\node[fill,circle,scale=.2] at (0,.3) {};
\node[fill,circle,scale=.2] at (0,0) {};
\node[fill,circle,scale=.2] at (0,-.3) {};
\node[fill,circle,scale=.5] at (1.2,.8) {};
\node[fill,circle,scale=.5] at (1.2,-.8) {};
\node[fill,circle,scale=.2] at (1.2,.3) {};
\node[fill,circle,scale=.2] at (1.2,0) {};
\node[fill,circle,scale=.2] at (1.2,-.3) {};
\node[right] at (1.2,0) {$n-r$ points};
\node[right] at (-2.5,0) {$r-1$ points};
\node[right] at (0,-1.25) {$c$};
\node[left] at (-3,0) {$C_{\gamma,r}[n]:$};
\end{tikzpicture}
\end{center}

\begin{example}
On $\mathbf{P}^2$, $C_{H,n-1}[n]$ is the curve given by varying one point on a line which contains $n-1$ fixed points.
\end{example}


\subsubsection{The Intersection Product}
The intersection product between these divisors and curves is well known, but we will recall it for completeness (c.f. \cite{ABCH}). 

For $i\neq0$, $C_i[n] \cdot H_j[n]$ is number of times the varying point intersects $h_j$ as the $n-1$ general points won't lie on $h_j$. 
Thus, $C_i[n] \cdot H_j[n] = H_i\cdot H_j$.
Similarly, $C_0[n] \cdot H_j[n] = H_1 \cdot H_j$.
Since $A[n]$ has fixed support, $A[n] \cdot H_j[n] = 0$. 

Again for $i\neq0$, each $C_i[n]$ contains no nonreduced schemes so $C_i[n] \cdot B[n] = 0$. However $C_0[n]$ has exactly one nonreduced scheme so $C_0[n] \cdot B[n] = 2$.
Lastly, in order to find $A[n] \cdot B[n]$, one can use test divisors to see that $A[n] = C_1[n] -C_0[n]$ (i.e. $A[n] \cdot H_i[n] =0$ implies $A[n] = k\left(C_1[n] -C_0[n]\right)$ and $A[n] \cdot D$ 
for any divisor $D$ involving $B$ determines $k=1$).
Using this class, $A[n] \cdot B[n] =-2$ \cite{ES}.
Intuitively, this intersection should be negative because every representative of $A[n]$ lives inside $B[n]$.

\begin{remark}
Note that if $H_i \cdot H_j =\delta_{ij}$, 
then $C_i[n]$ ($1 \leq i \leq k$) and $A[n]$ give a basis for the curves that has a diagonal intersection matrix with the given basis for the divisors.
\end{remark}

\begin{example}
If $X = \mathbf{P}^2$, there are only two divisors, $H[n]$ and $B[n]$, with three relevant curve classes, $C_0[n]$, $C_1[n]$, and $A[n]$. Given these divisors and curves, we have the following intersection product.
\begin{center}
  \begin{tabular}{|c|c|c|}
     \hline
                 & $H[n]$   & $B[n]$ \\ \hline
    $C_{0}[n]$   & $1$        & $2$ \\ \hline   
    $C_{1}[n]$   & $1$        & $0$ \\ \hline   
    $A[n]$       & $0$        & $-2$ \\ \hline   
  \end{tabular}
\end{center}
\end{example}


\subsection{Divisors and Curves on the Nested Hilbert Scheme}

The bases for divisors and curves of the Hilbert scheme generalize to bases of divisors and curves for the smooth nested Hilbert schemes.


\subsubsection{A Basis for the Divisors}

Intuitively, each divisor of $X$ now corresponds to two divisors.
One divisor imposes the analogous condition on the smaller subscheme while the other imposes it on the residual points.

Given the divisor $H_i$ on $X$ we have 
\[H^{\mathrm{diff}}_i=\{ (z,z')\in X^{[n+1,n]} 
: \mathrm{res}(z,z') \cap h_i \neq \emptyset\}  \text{, and}\]
\[ H^b_i=\{ (z,z')\in X^{[n+1,n]} 
: \text{supp}(z') \cap h_i \neq \emptyset\}   \]
We also define $H^a_i:=\mathrm{pr}_a^*(H) = H^{\mathrm{diff}}_i+H^b_i$. 

Interestingly, these geometrically defined divisors can also be realized as pull backs of the divisors along natural projections. 
\[ H^b_i=\mathrm{pr}_b^\ast \left(H_i[n]\right) \text{ and } H^{\mathrm{diff}}_i=\mathrm{res}^\ast \left(H_i\right).\]
Alternatively, we could have used the two pull backs of $H_i$ from the two Hilbert schemes as the two corresponding divisors, but the choice we made is more natural as these divisors are irreducible.

The only divisor left to generalize is $B$,
which also generalizes to two Cartier divisors on $X^{[n+1,n]}$.
The first divisor is where the residue of $(z,z')$ 
is part of the support of $z'$,
\[B^{\mathrm{diff}}[n+1,n]=\{ (z,z') \in X^{[n+1,n]} : \text{supp}(z') \cap \mathrm{res}(z,z') \neq \emptyset \}.\]
The second divisor is where the subscheme $z'$ is nonreduced,
\[B^b[n+1,n]=\{ (z,z') \in X^{[n+1,n]} : z' \text{ is nonreduced}\}.\]
Again, we can realize these as pullbacks, $B^b[n+1,n] = \mathrm{pr}_b^*\left( B[n]\right)$ and $B^{\mathrm{diff}}[n+1,n] =  \mathrm{pr}_a^*\left( B[n+1]\right)-\mathrm{pr}_b^*\left( B[n]\right)$. We also denote $B^a[n+1,n]=\mathrm{pr}_a^*\left( B[n+1]\right)$.

\begin{example}\label{ex:canonical divisor}
To be familiar with the notation, let's use Proposition \ref{prop:blow up of universal family} to give the class of the canonical divisor $K_{X^{[n+1,n]}}$ of $X^{[n+1,n]}$ when $X$ is the projective plane. Since $X^{[n+1,n]}$ is the blow up of $X^{[n]}\times X$ along $Z^{[n]}$, we have  
\begin{eqnarray*}
K_{\mathbf{P}^{2[n+1,n]}}&=&\mathrm{pr}_b^\ast(K_{\mathbf{P}^{2[n]}})+\mathrm{res}^\ast(K_{\mathbf{P}^2})+E\\
&=&-3H^b[n+1,n]-3H^{\mathrm{diff}}[n+1,n]+\frac{1}{2}B^{\mathrm{diff}}[n+1,n]
\end{eqnarray*}
\end{example}


\subsubsection{A Basis for the Curves}\label{sec:curvebasis}
The way that each curve class on $X^{[n]}$ generalizes to two classes on $X^{[n+1,n]}$ is by changing which subscheme contains the varying point.
We will use circles for points in $z'$ and squares for residual points. The shape is unfilled if the point is moving and
filled if the point is fixed. 

For $i\neq0$, let $C^a_i$ be the curve class of $(z,z')\subset X^{[n+1,n]}$ given by fixing a general $z'$ and varying the residual point of $z$ on $h_i$ for $1\leq i \leq k$.
\begin{center}
\begin{tikzpicture}[scale=1]
\draw plot [smooth, tension=1] coordinates { (-6.2,-1) (-6,-.6) (-6,.6) (-5.8,1)};
\node[draw=black,rectangle,scale=.8] at (-6,0) {};
\node[fill,circle,scale=.5] at (-5,.8) {};
\node[fill,circle,scale=.5] at (-5,-.8) {};
\node[fill,circle,scale=.2] at (-5,.3) {};
\node[fill,circle,scale=.2] at (-5,0) {};
\node[fill,circle,scale=.2] at (-5,-.3) {};
\node[right] at (-6,-1) {$h_i$};
\node[left] at (-7,0) {$C^a_i[n+1,n]:$};
\end{tikzpicture}
\end{center}

Let $C^b_i$ be the curve class given by fixing general points for all but one point of $z$, including the residual point, and varying the last point of $z' \subset z$ on $h_i$ for $1\leq i \leq k$. 

\begin{center}
\begin{tikzpicture}[scale=1]
\draw plot [smooth, tension=1] coordinates { (-6.2,-1) (-6,-.6) (-6,.6) (-5.8,1)};
\node[draw=black,circle,scale=.6] at (-6,0) {};
\node[fill,circle,scale=.5] at (-5,.8) {};
\node[fill,circle,scale=.5] at (-5,-.6) {};
\node[fill,circle,scale=.2] at (-5,.45) {};
\node[fill,circle,scale=.2] at (-5,.15) {};
\node[fill,circle,scale=.2] at (-5,-.15) {};
\node[fill,rectangle,scale=.5] at (-5,-1) {};
\node[right] at (-6,-1) {$h_i$};
\node[left] at (-7,0) {$C^b_i[n+1,n]:$};
\end{tikzpicture}
\end{center}

Next, let $C^a_0$ be the curve class given by fixing one point of $z'$ as a general point of $h_1$, fixing the rest of $z'$ as general points, and then varying the residual point of $z$ along $h_1$.

\begin{center}
\begin{tikzpicture}[scale=1]
\draw plot [smooth, tension=1] coordinates { (-6.2,-1) (-6,-.6) (-6,.6) (-5.8,1)};
\node[draw=black,rectangle,scale=.7] at (-6,-.25) {};
\node[fill,circle,scale=.5] at (-6,.25) {};
\node[fill,circle,scale=.5] at (-5,.8) {};
\node[fill,circle,scale=.5] at (-5,-.8) {};
\node[fill,circle,scale=.2] at (-5,.3) {};
\node[fill,circle,scale=.2] at (-5,0) {};
\node[fill,circle,scale=.2] at (-5,-.3) {};
\node[right] at (-6,-1) {$h_i$};
\node[left] at (-7,0) {$C^a_0[n+1,n]:$};
\end{tikzpicture}
\end{center}

Let $C^b_0$ be the curve class given by fixing $n-2$ general points of $X$ as a subset of $z'$, fixing one general point of $h_1$ as a subset of $z'$, fixing a general point of $X$ as the residual point, and then varying $n$-th point of $z'$ along $h_1$. 
\begin{center}
\begin{tikzpicture}[scale=1]
\draw plot [smooth, tension=1] coordinates { (-.2,-1) (0,-.6) (0,.6) (.2,1)};
\node[draw=black,circle,scale=.5] at (0,-.25) {};
\node[fill,circle,scale=.5] at (0,.25) {};
\node[fill,circle,scale=.5] at (1.2,.8) {};
\node[fill,circle,scale=.5] at (1.2,-.4) {};
\node[fill,circle,scale=.2] at (1.2,.5) {};
\node[fill,circle,scale=.2] at (1.2,0.2) {};
\node[fill,circle,scale=.2] at (1.2,-.1) {};
\node[fill,rectangle,scale=.5] at (1.2,-.9) {};
\node[right] at (0,-1) {$h_1$};
\node[left] at (-1,0) {$C^b_0[n+1,n]:$};
\end{tikzpicture}
\end{center}


Finally, let $ A^a$ be defined as the curve class given by fixing a general $z'$ and then varying the double point structure of the residual point of $z$ over a fixed point of $z'$.
\begin{center}
\begin{tikzpicture}
\node[fill,circle,scale=.5] at (0,0) {};
\draw[dashed,thick] (0,0) -- (1/2,1/2);
\node[draw=black,rectangle,scale=.7] at (1/2,1/2) {};
\node[fill,circle,scale=.5] at (1.4,.8) {};
\node[fill,circle,scale=.5] at (1.4,-.8) {};
\node[fill,circle,scale=.3] at (1.4,.3) {};
\node[fill,circle,scale=.3] at (1.4,0) {};
\node[fill,circle,scale=.3] at (1.4,-.3) {};
\node[left] at (-1,0) {$ A^a[n+1,n]:$};
\end{tikzpicture}
\end{center}

Similarly, let $ A^b$ be defined as the curve class given by fixing a scheme $z$ supported on $n-1$ general points with length three at the support of the residual point and varying $z'$ by fixing the reduced points and varying the double scheme structure at the nonreduced point.

\begin{center}
\begin{tikzpicture}
\node[fill,circle,scale=.5] at (0,0) {};
\draw[-{>[scale=1]},dashed,thick] (0,0) -- (1/2,1/2);
\draw[-,thick] (0,0) -- (0,.71);
\node[fill,circle,scale=.5] at (1.4,.8) {};
\node[fill,circle,scale=.5] at (1.4,-.5) {};
\node[fill,circle,scale=.2] at (1.4,.4) {};
\node[fill,circle,scale=.2] at (1.4,0.1) {};
\node[fill,circle,scale=.2] at (1.4,-.2) {};
\node[fill,rectangle,scale=.56] at (0,.7) {};
\node[left] at (-1,0) {$ A^b[n+1,n]:$};
\end{tikzpicture}
\end{center}

\begin{remark}
Note that for $0\leq i \leq k$, we have
\[ \mathrm{pr}_a(C^a_i[n+1,n])=C_i[n+1] \hspace{.2in}\text{  and  }\hspace{.2in} \mathrm{pr}_b(C^b_i[n+1,n])=C_i[n]. \]

We have similar descriptions of the images of the $A$ curves.
\[\mathrm{pr}_a(A^a[n+1,n]) = A[n+1], \hspace{.5in} \mathrm{pr}_b(A^a[n+1,n]) =\text{pt}, \]
\[\mathrm{pr}_a(A^b[n+1,n]) = \text{pt}, \hspace{.3in}\text{ and } \hspace{.3in} \mathrm{pr}_b(A^b[n+1,n]) =A[n] \]
These equalities will help to compute the intersection products.
\end{remark}

\subsection{Other curves and divisors}
For our computations, we will need a few other curve and divisor classes. 
These are two of the three natural generalizations of the other curves in the Hilbert scheme case.

\subsubsection{Divisors}
Given a divisor on $X^{[n]}$, we can write some multiple of it as 
\[ D_{\mathbf{m}}[n] = \sum_{i=1}^k m_i \cdot H_i[n] -\frac{1}{2}B[n]. \]
or as $\sum_{i=1}^k m_i \cdot H_i[n]$ where $\mathbf{m}=\{m_1,\ldots, m_k \}$ .
We will denote the pullback of $D_{\mathbf{m}}[n+1]$ to the nested Hilbert scheme $X^{[n+1,n]}$ along $\mathrm{pr}_a$ by $D^a_{\mathbf{m}}$: 
\[D^a_\mathbf{m} := \mathrm{pr}_a^*\left(D_{\mathbf{m}}\right)= m\cdot \left( B^a \right) + \sum_{i=1}^k m_i \cdot H^a_i\]

Similarly, we will denote its pullback to the nested Hilbert scheme along $\mathrm{pr}_b$ by $D^b_{\mathbf{m}}$:    
\[D^b_{\mathbf{m}} := \mathrm{pr}_b^*\left(D_{\mathbf{m}}\right)= m\cdot B^b + \sum_{i=1}^k m_i\cdot H^b_i.\]

\begin{example}\label{ex:D_n}
On $\mathbf{P}^2$, $D_k[n] = k H[n] -\frac{1}{2}B[n]$ so 
\[D^a_{k} = k H^a -\frac{1}{2}B^a.\] 
\end{example}

\subsubsection{Curves}\label{sec:othercurve}
First let 
\[ \gamma=\sum_{i=1}^k m_i \cdot h_i \]
be a curve class on $X$ and let $c$ be a general irreducible curve of class $\gamma$.

Let $C^a_{\gamma,r}$ be the curve class of $(z,z')\subset X^{[n+1,n]}$ given by fixing $r-1$ general points of $c$ as part of $z'$, $n+1-r$ general points not contained in $c$ as part of $z'$, with $x=\mathrm{res}(z,z')$ varying on $c$ for $1 \leq r \leq n+1.$
\begin{center}
\begin{tikzpicture}[scale=1]
\draw plot [smooth, tension=.5] coordinates { (-.2,-1.5) (0,-1.1) (0,.6) (.2,1)};
\node[draw=black,rectangle,scale=.5] at (0,-1.1) {};
\node[fill,circle,scale=.5] at (0,.6) {};
\node[fill,circle,scale=.5] at (0,-.6) {};
\node[fill,circle,scale=.2] at (0,.25) {};
\node[fill,circle,scale=.2] at (0,0) {};
\node[fill,circle,scale=.2] at (0,-.25) {};

\node[fill,circle,scale=.5] at (1.2,.8) {};
\node[fill,circle,scale=.5] at (1.2,-.8) {};
\node[fill,circle,scale=.2] at (1.2,.3) {};
\node[fill,circle,scale=.2] at (1.2,0) {};
\node[fill,circle,scale=.2] at (1.2,-.3) {};
\node[right] at (1.2,0) {$n-r+1$ points};
\node[right] at (-2.5,0) {$r-1$ points};
\node[right] at (0,-1.25) {$c$};
\node[left] at (-3,0) {$C^a_{\gamma,r}[n+1,n]:$};
\end{tikzpicture}
\end{center}


Let $C^b_{\gamma,r}$ be the curve class given by fixing $r-1$ general points of $c$ as part of $z'$, $n-r$ general points not contained in $c$ as part of $z'$, fixing $x=\mathrm{res}(z,z') \in c$, and varying the last point of $z'$ on $c$.
\begin{center}
\begin{tikzpicture}[scale=1]
\draw plot [smooth, tension=.5] coordinates { (-.2,-1.5) (0,-1.1) (0,.6) (.2,1)};
\node[draw=black,circle,scale=.5] at (0,-1.1) {};
\node[fill,circle,scale=.5] at (0,.6) {};
\node[fill,circle,scale=.5] at (0,-.4) {};
\node[fill,circle,scale=.2] at (0,.3) {};
\node[fill,circle,scale=.2] at (0,.1) {};
\node[fill,circle,scale=.2] at (0,-.1) {};
\node[fill,rectangle,scale=.5] at (0,-.75) {};

\node[fill,circle,scale=.5] at (1.2,.8) {};
\node[fill,circle,scale=.5] at (1.2,-.8) {};
\node[fill,circle,scale=.2] at (1.2,.3) {};
\node[fill,circle,scale=.2] at (1.2,0) {};
\node[fill,circle,scale=.2] at (1.2,-.3) {};
\node[right] at (1.2,0) {$n-r$ points};
\node[right] at (-2.5,0) {$r-1$ points};
\node[right] at (0,-1.25) {$c$};
\node[left] at (-3,0) {$C^b_{\gamma,r}[n+1,n]:$};
\end{tikzpicture}
\end{center}

The classes of these curves can be expressed as
\[C^a_{\gamma,r} = \sum_{i=1}^k m_i\cdot C^a_i-\left(r-1\right) A^a \text{ and } C^b_{\gamma,r} = \sum_{i=1}^k m_i\cdot C^b_i-A^a-\left(r-1\right) A^b.\]

\begin{remark}
Note, as before we have
\[ \mathrm{pr}_a(C^a_{\gamma,r})=C_{\gamma,r}[n+1] \text{, }\hspace{.8in} \mathrm{pr}_b(C^a_{\gamma,r})=\text{pt},\]
\[ \mathrm{pr}_a(C^b_{\gamma,r})=C_{\gamma,r+1}[n+1], \hspace{.3in} \text{  and  }\hspace{.3in} \mathrm{pr}_b(C^b_{\gamma,r})=C_{\gamma,r}[n].\]
\end{remark}




\subsubsection{The Intersection Product}
We now want to compute the intersections between our two bases.

We only consider the intersections of $C^a_j$ and $C^b_j$ because the intersections of $ A^a$ and $A^b$ follow from using test divisors to compute their classes as
\[ A^a = C^a_1-C^a_0 \text{ and } A^b = C^b_1-C^b_0-(C^a_1-C^a_0).\]

\textit{(1) Intersection with $H^{\mathrm{diff}}_i$.} 
As we are varying a residual point of $z$ along $h_j$ on $C^a_j$, by the projection formula, we have $H^\mathrm{diff}_i \cdot C^a_j =H_i\cdot H_j$.
Similarly, the residual points of $z$ are fixed on $C^b_j$, so $H^\mathrm{diff}_i \cdot C^b_j= 0$. 

\textit{(2) For intersection with $H^b_i$.} 
Because $z'$ is fixed on $C^a_j$,  we have $H^b_i\cdot C^a_j=0$. 
Similarly, since $\mathrm{pr}_b(C^b_j)=C_j[n]$ and $\mathrm{pr}_b(H^b_i)=H_i$, we have $H^b_i \cdot C^b_j=H_i\cdot H_j$.
  
\textit{(3) Intersection with $B^{\mathrm{diff}},B^b$.} 
When $1\leq j \leq k$, $C^a_j$ and $C^b_j$ have no nonreduced schemes on them so $C^a_j \cdot B^{\mathrm{diff}} =C^a_j \cdot B^b =C^b_j \cdot B^{\mathrm{diff}} =C^b_j \cdot B^b =0$.
When $j=0$, however, $C^a_0$ has a single nonreduced scheme where the residual point is supported on a point of $z'$ so $C^a_0 \cdot B^{\mathrm{diff}} =2, C^a_0 \cdot B^b =0$.
Similarly, $C^b_0$ has a single nonreduced scheme where the nonreduced scheme is a subscheme of $z'$; hence, $C^b_0 \cdot B^{\mathrm{diff}} =0, C^b_0 \cdot B^b =2$.

\begin{example}
For the case of $\mathbf{P}^2$, we have four divisors and the following intersection product where we have writen $H_1$ as $H$ as there is only one ample divisor. 
\begin{center}
  \begin{tabular}{|c|c|c|c|c|}
     \hline
            & $H^\mathrm{diff}$ & $B^{\mathrm{diff}}$ & $H^b$ & $B^b$ \\ \hline
    $C^a_0$ & 1   & 2     & 0     & 0   \\ \hline      
    $C^a_1$ & 1   & 0     & 0     & 0  \\ \hline
    $A^a$     & 0   & -2    & 0     & 0  \\ \hline
    $C^b_0$ & 0   & 0     & 1     & 2   \\ \hline
    $C^b_1$ & 0   & 0     & 1     & 0   \\ \hline   
    $A^b$     & 0   & 2     & 0     & -2 \\ \hline
  \end{tabular}
\end{center}
\end{example}


\section{The Nef Cone}
\label{sec: nef}
We now turn to computing the nef cone of each smooth nested Hilbert scheme on the projective plane, a Hirzebruch surface, or a K3 surface of Picard rank one.


\subsection{Cones and the cone dualities}
First,  we recall some basic facts from \cite[Section 1.4]{Lazarsfeld1}. 
For a projective algebraic variety $X$, the nef cone of divisors $\text{Nef}(X)\subset N^1(X)_\mathbf{R}$ can be defined as follows:
\[\text{Nef}(X) = \{C\cdot D \geq 0 \text{ for all irreducible curves } C \subset X\}.\]

If $C\cdot D =0$, we say that $C$ is dual to $D$. Our strategy for computing nef cones is to find sufficiently many nef divisors and effective curves dual to them with an intersection matrix that is diagonal. 
It follows that these divisors span the entire nef cone.

In this section we will proceed in three steps.
First, we define the needed notation.
Then, we give the nef cones of 
$X^{[n]}$ for these surfaces given by \cite[Theorem 3.14]{LQZ}, \cite[Theorem 2.4]{BC} and \cite[Proposition 10.3]{BM2}, respectively.
Finally, we use these to construct the nef cone of $X^{[n+1,n]}$ for these surfaces.

Let us fix some notation for the surfaces we will use.

\begin{enumerate}
   \item Denote the projective plane by $\mathbf{P}^2$. Denote the class of a line by $H$. Note that $H^2 =1$.
   \item Denote the $i$-th Hirzebruch surface by $\mathbf{F}_i$. Denote the classes of a hyperplane section and a fiber of the ruling, respectively by $H$ and $F$. Note that $H^2 = i$, $F^2=0$, and $H\cdot F = 1$. Also in the case of $\mathbb{F}_0$, $F$ is a line of the other ruling and $H$ is a hyperplane section of that ruling. In order to be consistent with defining our curve classes, we define $H=H_1$ and $F =H_2$. 
   \item Denote a general genus $g$ K3 surface by $S$. Denote the class of a hyperplane section by $H$. Notice that $H^2= 2g-2$.
\end{enumerate}

\subsection{Nef Cone of the Hilbert Scheme}
Now let us recall the nef cones of $X^{[n]}$ on each surface.
Although some of these are well known to experts, we have included them for completeness.
We will also give a brief outline of the argument in each case.

\subsubsection{The projective plane}
\cite{LQZ} showed that the nef cone of $\mathbf{P}^{2[n]}$ is spanned by 
\[ H[n] \text{ and } D_{n-1}[n] =(n-1)H[n]-\frac{1}{2}B[n]. \]
A representative of $D_{n-1}[n]$ can be described as the schemes which lie on a curve of degree $n-1$ with $\binom{n}{2}$ fixed points. Since no collection of points of $\mathbf{P}^2$ is on every line, $H[n]$ is basepoint-free; hence, it is nef.
As any collection of $n$ points on $\mathbf{P}^2$ lies on some curve of degree $n-1$ and off of one, $D_{n-1}[n]$ is base point free; hence, it is nef.

Using curves defined in (\ref{section:other curves}), we have the following intersection table which completes the calculation of nef cone.

\begin{center}
  \begin{tabular} {| c | c | c |}
    \hline
                  & $H[n]$ & $D_{n-1}[n]$  \\ \hline 
    $C_{H,n}[n]$ & 1 & 0  \\ \hline
    $A[n]$         & 0 & 1 \\ 
    \hline
  \end{tabular}
\end{center}

\subsubsection{Hirzebruch surfaces}

\cite{BC} also showed that the nef cone of $\mathbf{F}_i^{[n]}$ for $i>0$ is spanned by 
\[ H[n], F[n] \text{ and } D_{n-1,n-1}[n] = (n-1)H[n] + (n-1)F[n] -\frac{B}{2}[n]. \] 
A representative of $D_{n-1,n-1}[n]$ can again be described as the schemes which lie on a curve of type $(n-1)H+(n-1)F$ with the correct number of fixed points. 

By similar reasoning, these three divisors are basepoint free, hence nef. 

Again we have a three dimensional convex cone so we need curves dual to each pair of sides. Recalling $H$ and $F$ were a fiber and a section of the ruling, respectively, we see that $C_{H+iF,n}[n]$, $C_{F,n}[n]$, and $A[n]$ are the three needed dual curves.

\subsubsection{Hilbert schemes of points on K3 surfaces for high numbers of points.}
The nef cones of these Hilbert schemes are much more subtle than those on rational surfaces.
$k$-very ample line bundles are no longer sufficient to span the nef cone.
\cite{BM2} computed this nef cone using Bridgeland stability .
However, $n$ needs to be sufficiently large for the cleanest solution, and we now add that as a standing assumption.
We denote the extremal nef class they construct as $D_n$. Then the nef cone of $S^{[n]}$ is spanned by
\[ H[n] \text{ and } D_n=\frac{n-1+g}{2g-2}H[n]-\frac{1}{2}B[n].\]

Recall, that $H$ is the generator of the Picard group on $S$.
Then, to see that $H[n]$ and $D_n$ span the nef cone, it suffices to find a dual curve to $D_n$.
That dual curve is given by the fibers of a $g^1_n$, i.e. the fibers of an $n$ to $1$ map from a curve with class $H$ to $\mathbb{P}^1$.
As the fibers of the $g^1_n$ exactly cover a fixed curve with class $H$, its intersection with the divisor $H[n]$ is $2g-2$.
By the adjunction formula, its intersection with $B[n]$ is $g-1+n$.
This shows that the non-trivial edge of the nef cone is spanned by the class $\frac{n-1+g}{2g-2}H[n]-\frac{1}{2}B[n]$.

Thus, we have the following intersection table which completes the calculation.
\begin{center}
  \begin{tabular} {| c | c | c |}
    \hline
                  & $H[n]$ & $D_{n}$  \\ \hline 
    $g^1_n$       & $2g-2$   & 0        \\ \hline
    $A[n]$        & 0      & 1        \\ \hline
  \end{tabular}
\end{center}

\begin{remark}

It is useful for us to point out an alternative construction of the dual curve.
By \cite[Appendix]{MM2}, there exist rational $g$-nodal curves with class $H$.
We then construct a curve $C_{\text{nodal}}$ on $S^{[n]}$, $n>g$, by fixing $g$ points of the scheme on the nodes, fixing $n-1-g$ points of the scheme as general points of the curve, and varying the $n$-th point along the curve.

As the varying point hits another reduced point of the scheme exactly $(n-1-g)+2g=n-1+g$ times,  $C_{\text{nodal}} \cdot B[n] =2(n-1+g)$. As the point varies on a curve of class $H$, $C_{\text{nodal}} \cdot H[n] = 2g-2$.
Together, these intersections show that $C_{\text{nodal}}$ is also a dual curve to the non-trivial edge of the nef cone.
This construction generalizes more easily to nested Hilbert schemes.
\end{remark}


\subsection{Nef Cone of the Nested Hilbert Scheme}
We now turn to computing the nef cone of $X^{[n+1,n]}$ where $X$ is $\mathbf{P}^2$, $\mathbf{F}_i$, or $S$.


By Proposition \ref{prop:Picard group}, the Picard group is
\[ \text{Pic}(X^{[n+1,n]}) \cong \mathrm{pr}_a^\ast \left(\text{Pic}(X^{[n+1]}) \right)\oplus \mathrm{pr}_b^\ast \left(\text{Pic}(X^{[n]})\right)\]
It will turn out that similar statement almost holds for the nef cone, as the spanning nef divisors will be (differences of) pull backs of extremal nef divisors on the Hilbert schemes.

Take $\mathbf{P}^2$ as an example: intuitively, each dual curve should project either to a dual curve or to a point so we consider $A^a$, $A^b$, $C^a_{H,n+1}$, and $C^b_{H,n}$.
Pairing these against the four pull backs of the extremal rays of the nef cones, we get the following.

\begin{center}
  \begin{tabular} {|c | c | c | c | c|}
    \hline
                  & $H^b$ & $D^b_{n-1}$ & $H^{a}$ & $D^a_{n}$  \\ \hline 
    $C^b_{H,n}$   & 1 & 0 & 1 & 0 \\ \hline
    $A^b$         & 0 & 1 & 0 & 0 \\ \hline
    $C^a_{H,n+1}$ & 0 & 0 & 1 & 0 \\ \hline
    $A^a$         & 0 & 0 & 0 & 1 \\ 
    \hline
  \end{tabular}
\end{center}
From this, we immediately see that three of those pull backs are in fact extremal rays as they are each the pull back of a nef line bundle along a morphism which is dual to three independent effective curve classes. 

The question then is how to correct the fourth extremal ray.
Ideally, the intersection matrix would be diagonal so we consider replacing $H^a$ with $H^a-H^b$.
The difference is just $H^{\mathrm{diff}}$ and it actually works.


We will compute the nef cones of $\mathbf{P}^2$, $\mathbf{F}_0 =\mathbf{P}^1 \times \mathbf{P}^1$, and $\mathbf{F}_i$, $i>0$ separately for notational reasons, but the proofs are entirely analogous.

\begin{prop}\label{prop:nested hilbert scheme}
$\text{Nef}\left(\mathbf{P}^{2}\right)^{[n+1,n]}$ is spanned by 
\[ H^b[n+1,n], H^{\mathrm{diff}}[n+1,n], D^b_{n-1}\text{, and } D^a_{n}\]
\end{prop}

\begin{proof}
We drop $[n+1,n]$ from the notation of divisors through the proof for convenience. 
We first show that each of these four divisors is nef.
Recall that $H^b$, $D^a_{n}$, and $D^b_{n-1}$ are pull backs along the projections $\mathrm{pr}_a$ and $\mathrm{pr}_b$ of nef divisors $H[n]$, $nH[n+1]-\frac{1}{2}B[n+1]$, and $(n-1)H[n]-\frac{1}{2}B[n]$, respectively.
As pullbacks of nef divisors along a morphism, they are nef. 
$H^{\mathrm{diff}}$ is nef as it is basepoint free since no residual point lies on every line of $\mathbf{P}^2$ .
Equivalently, $H^{\mathrm{diff}}$ is basepoint free as it is the pull back of a base point free divisor along the residue map.

Since all four divisors are nef, it suffices to bound the nef cone using four irreducible curves that each are dual to three of the divisors and intersect the fourth divisor positively.
The curves we will need are $A^a$, $A^b$, $C^a_{H,n+1}$ and $C^b_{H,n}$ (see their construction in Section (\ref{sec:curvebasis}) and (\ref{sec:othercurve})).
The necessary intersections to conclude the proof are summarized in the following table.
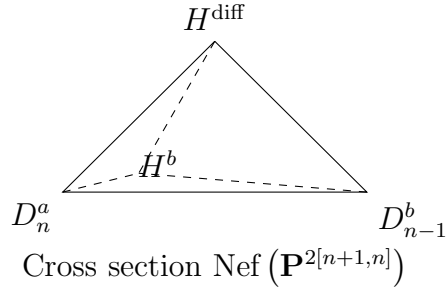
\begin{center}
  \begin{tabular} {|c | c | c | c | c|}
    \hline
                  & $H^b$ & $D^b_{n-1}$ & $H^{\mathrm{diff}}$ & $D^a_{n}$  \\ \hline 
    $C^b_{H,n}$   & 1 & 0 & 0 & 0 \\ \hline
    $A^b$         & 0 & 1 & 0 & 0 \\ \hline
    $C^a_{H,n+1}$ & 0 & 0 & 1 & 0 \\ \hline
    $A^a$         & 0 & 0 & 0 & 1 \\ 
    \hline
  \end{tabular}
\end{center}

\end{proof}

\begin{center}
\begin{tikzpicture}[scale=.5]
\draw (0,4) -- (-4,0) -- (4,0) -- (0,4);
\draw[dashed] (0,4) -- (-2,1/2);
\draw[dashed] (4,0) -- (-2,1/2);
\draw[dashed] (-4,0) -- (-2,1/2);
\node[above] at (0,4) {$H^{\mathrm{diff}}$};
\node[below left] at (-4,0) {$D^a_{n}$};
\node[below right] at (4,0) {$D^b_{n-1}$};
\node at (-6/4,3/4) {$H^b$};
\node at (0,-2) {Cross section $\text{Nef}\left( \mathbf{P}^{2[n+1,n]}\right)$};
\end{tikzpicture}
\end{center}

The proofs of the remaining nef cones follow the same pattern.
We only give the reason why each divisor is nef and the set of dual curves.

\begin{prop}
$\text{Nef}\left(\mathbf{F}_i^{[n+1,n]}\right)$ is spanned by $H^{\mathrm{diff}}$, $H^b$, $F^{\mathrm{diff}}$, $F^b$, $D^b_{(n-1,n-1)}$, and $D^a_{(n,n)}$.
\end{prop}

\begin{proof}
$H^b$, $F^b$, $D^a_{(n,n)}$, and $D^b_{(n-1,n-1)}$ are pull backs of nef divisors along the projections to the respective Hilbert schemes so are nef.
$H^{\mathrm{diff}}$ and $F^{\mathrm{diff}}$ are basepoint free as every residual point lies on only two dimensions of hyperplane sections and on a unique line of ruling.
The curves we will need are $A^a$, $ A^b$, $C^a_{H+iF,n+1}$, $C^b_{F,n+1}$, $C^b_{F,n}$ and $C^b_{H+iF,n}$.
\end{proof}

\subsubsection{K3 surfaces}
As in the previous cases, we will see that we can realize divisors spanning the four extremal rays as pull backs of nef divisors from our two projections and the residual map.
The nonreduced dual curves are the same in this case, but we have to construct two different generalizations of the dual curve to the Hilbert scheme using our alternative construction of it.

\begin{thm}
Let $S$ be a Picard rank one K3 surface of genus $g >2$, $n\geq g+1$.
The nef cone of $S^{[n+1,n]}$ is spanned by $H^{\mathrm{diff}}$, $H^b$, $D^a_{f(n+1)}$, and $D^b_{f(n)}$ where $f(m)=\frac{m-1+g}{2g-2}$.
\end{thm}

\begin{proof}
We first see that each of these divisors is nef as the pull back of a nef divisor along a morphism.
$H^{\mathrm{diff}}$ is the pull back of the ample divisor $H$ along the residual morphism, $H^b$ is the pull back of nef divisor $H[n]$ along $\mathrm{pr}_b$, $D^a_{f(n+1)}$ is the pull back of nef divisor $D_{f(n+1)}$ along $\mathrm{pr}_a$, and $D^b_{f(n)}$ is the pull back of the nef divisor $D_{f(n)}$ along $\mathrm{pr}_b$.

To prove the statement, it now suffices to find four curves which are each dual to three of these divisors.
Note $D^a_{f(n+1)} = \frac{n+g}{2g-2}H^{\mathrm{diff}}+\frac{n+g}{2g-2}H^{b}-\frac{1}{2}B^{\mathrm{diff}}-\frac{1}{2}B^b$ and $D^b_{f(n+1)} = \frac{n-1+g}{2g-2}H^{b}-\frac{1}{2}B^b$.

$A^a$ and $A^b$ are two of those curves and we now construct the remaining two.

Recall that by \cite{MM2}, there exist $g$-nodal rational nodal curves with class $H$ on $X$.
We then construct a curve $C^a_{\text{nodal}}$ on $S^{[n+1,n]}$, $n>g$, by fixing $g$ points of the subscheme on the nodes, fixing $n-g$ points of the subscheme as general points of the curve, and varying the residual point along the curve.
Similarly, we construct a curve $C^b_{\text{nodal}}$ on $S^{[n+1,n]}$, $n>g$, by fixing $g$ points of the subscheme on the nodes, fixing $n-1-g$ points of the subscheme as general points of the curve, fixing the residual point as a general point of the curve, and varying the $n$-th point of the subscheme along the curve.

We want to compute the intersection of these curves with the divisors.
Let us start with the $H$ divisors.
As the curve the residual point varies on has class $H$, $C^a_{\text{nodal}} \cdot H^{\mathrm{diff}} = 2g-2$ and as the subscheme is fixed, $C^a_{\text{nodal}} \cdot H^b = 0$.
As the residual point is fixed, $C^b_{\text{nodal}} \cdot H^{\mathrm{diff}} = 0$ and as the $n$-th subscheme point varies on a curve of class $H$, $C^b_\text{nodal} \cdot H^b = 2g-2$.

By the projection formula, $C^a_{\text{nodal}} \cdot \frac{B^a}{2}=n-g+2g=n+g$ and $C^a_{\text{nodal}} \cdot \frac{B^b}{2}=0$, so that $C^a_{\text{nodal}} \cdot \frac{B^{\mathrm{diff}}}{2} =n+g$.
Similarly, $C^b_{\text{nodal}} \cdot \frac{B^a}{2}=n+g$ and $C^a_{\text{nodal}} \cdot \frac{B^b}{2}=n-1+g$, so that $C^a_{\text{nodal}} \cdot \frac{B^{\mathrm{diff}}}{2} =1$.

The intersections of divisors with $A^a$ and $A^b$ were established in Section \ref{sec: div and curves}, and we just established the intersections with $C^a_\text{nodal}$ and $C^b_\text{nodal}$.
Thus, the four needed dual curves are $A^a$, $A^b$, $C^a_\text{nodal}$, and $C^b_\text{nodal}$, and the intersection table below completes the proof.

\begin{center}
  \begin{tabular} {|c | c | c | c | c|}
    \hline
                  & $H^b$ & $D^b_{f(n)}$ & $H^{\mathrm{diff}}$ & $D^a_{f(n+1)}$  \\ \hline 
    $C^b_\text{nodal}$   & $2g-2$ & 0 & 0 & 0 \\ \hline
    $A^b$         & 0 & 1 & 0 & 0 \\ \hline
    $C^a_\text{nodal}$ & 0 & 0 & $2g-2$ & 0 \\ \hline
    $A^a$         & 0 & 0 & 0 & 1 \\ 
    \hline
  \end{tabular}
\end{center}
\end{proof}

\begin{remark}
The proof for the nef cone is the same for any Hilbert scheme of points ($n>>0$) on any smooth Picard rank one surface with no irregularity by \cite{BHLRSWZ} as it is for K3 surfaces. The obstacle to generalizing the proof for the nested Hilbert scheme from K3 surfaces to those surfaces is the construction of the dual curve. In order to generalize that curve in the same fashion as we have, there would need to be a $\frac{(K_S+aH)aH}{2}+1$ nodal curve with class $aH$ (using the notation of \cite{BHLRSWZ}). We do now know if this is possible.
\end{remark}
\section{Nef Cone of the Universal Family}
\label{sec: univ fam}
The universal family case follows almost the exact same template so we cover it briefly.

\subsection{The relevant bases}

\subsubsection{A basis for divisors}
On the universal family, \cite[Theorem 7.6]{Fogarty2} showed that the Picard group is spanned by $H^{\mathrm{diff}}$ and $H^b$ divisors, analogous to those we listed on $X^{[n+1,n]}$, and that the locus of nonreduced schemes, which is a $\mathbf{Q}$-Cartier divisor with two irreducible Weil divisor components neither of which is $\mathbf{Q}$-Cartier.
These components precisely characterize the failure of the universal family to be $\mathbf{Q}$-factorial.
This whole locus is the pull back of $B$ from the projection to the Hilbert scheme,
\[B[n,1] =\{ (z,z') \in X^{[n,1]} : z \text{ is nonreduced}\} = \mathrm{pr}_a^*\left(B[n]\right).\]

\subsubsection{A basis for the curves}
For $i\neq0$, let $C^a_i$ be the curve class of $(z,z')\subset X^{[n,1]}$ given by fixing a general $z'$, fixing all but one point of $z$, and varying the last point of $z$ on $h_i$ for $1\leq i \leq k$.

Let $C^b_i$ be the curve class given by fixing the residual points as general points of $X$ and varying $z'$ on $h_i$ for $1\leq i \leq k$. 

Next, let $C^a_0$ be the curve class given by fixing $z'$ as a general point of $h_1$, fixing all but one point of $z$ as general points, and then varying the last point of $z$ along $h_1$.

Finally, let $ A^a$ be defined as the curve class given by fixing all but one point of $z$, including the residual point, as general points and varying the double point structure of the remaining point of $z$ over $z'$.

\begin{remark}
Note that on the universal families, we have
\[ \mathrm{pr}_a(C^a_i[n,1])=C_i[n] \text{ and } \mathrm{pr}_b(C^b_i[n,1])=H_i.\]

Similarly, for the $A$ curves, we have
\[\mathrm{pr}_a(A^a[n,1]) = A[n] \text{ and }  \mathrm{pr}_b(A^a[n,1]) =\mathrm{pt}. \]
\end{remark}

\subsubsection{The intersection product}
The intersection of each of these curves with each divisor is entirely analogous to the smooth case so we omit it.

\subsubsection{Other curves}
We now give the remaining analogous curves on the universal family.
Again, let 
\[ \gamma=\sum_{i=1}^k m_i \cdot h_i \]
be a curve class on $X$ and let $c$ be a general irreducible curve of class $\gamma$.

Define $C^a_{\gamma,r}$ to be the curve class of $(z,z')\subset X^{[n,1]}$ given by fixing a general $z'$ on $c$, fixing $r-1$ general points of $c$ as part of $z$, $n-r-2$ general points of $z$ not contained in $c$, and varying the last point of $z$ on $c$ for $1 \leq r \leq n-1$.

Let $C^b_{\gamma,r}$ be the curve class given by fixing $r-1$ general points of $z$ on $c$, $n-r$ general points of $z$ not contained in $c$, and varying $z'$ on $c$ for $1 \leq r \leq n$.

The classes of these curves can be computed to be
\[C^a_{\gamma,r} = \sum_{i=1}^k m_i\cdot C^a_i-\left(r-1\right) A^a  \text{  and  } C^b_{\gamma,r} = \sum_{i=1}^k m_i\cdot C^b_i-\left(r-1\right) A^c.\]

\subsection{Nef Cone of the Universal Family}
The nef cone of $Z^{[n]}$ on these surfaces is computed similarly to that of $X^{[n+1,n]}$; it differs only due to the drop in Picard rank.
\begin{prop}\label{prop:nef cone universal family}
$\text{Nef}\left(\mathbf{P}^{2[n,1]}\right)$ is spanned by $H^{\mathrm{diff}}$, $H^b$, and $D^a_{n-1}$.
\end{prop}

\begin{proof}
All three divisors are the pull backs of nef divisors along the projections, so they are nef. 
The dual curves we need are $A^a$, $C^a_{l,n-1}$ and $C^b_{l,n}$.
\end{proof}

\begin{prop}\label{prop:universal family of Hirzebruch surface}
$\text{Nef}\left(\mathbf{F}_i^{[n,1]}\right)$ is spanned by $H^{\mathrm{diff}}$, $H^b$, $F^{\mathrm{diff}}$, $F^b$, and $D^a_{(n-1,n-1)}$.
\end{prop}

\begin{proof}
All five divisors are the pull backs of nef divisors along the projections, so they are nef. 
The dual curves we need are $ A^a$, $C^a_{H+iF,n-1}$, $C^a_{F,n-1}$, $C^b_{H+iF,n}$ and $C^b_{F,n}$.
\end{proof}

%
%

\begin{thm}
Let $X$ be a Picard rank one K3 surface of genus $g\geq 2$, $n\geq g+1$.
The nef cone of $X^{[n,1]}$ is spanned by $H^{\mathrm{diff}}$, $H^b$ and $D^a_{f(n)}$ where $f(m)=\frac{m-1+g}{2g-2}$.
\end{thm}

\begin{proof}
All three divisors are the pull backs of nef divisors along the projections, so they are nef. 
The dual curves we need are $A^a$, $C^a_\text{nodal}$ and $C^b_\text{nodal}$.
\end{proof}

\section{Application}
\label{sec: application}

This section is devoted to an application to the questions of syzygies of line bundles on projective surfaces. Let $X$ be a smooth projective surface and $A$ is an ample line bundle on $X$. Consider the adjoint linear series of multiples of $A$, i.e. $L=K_X+nA$. As a corollary of Reider's theoerm \cite[Corollary 2.7]{Lazarsfeld3}, $L$ is very ample as long as $n \geq 4$. Hence, $L$ embeds $S$ into $\mathbf{P}H^0(L)$ as a projective subvariety. Starting from \cite{Green1}, people are interested in the smallest $n$ such that $L$ is projectively normal, i.e. the homogeneous coordinate ring of $X$ is normal. In other words, the multiplication map $H^0(L)\otimes H^0(kL) \to H^0((k+1)L)$ is surjective for every positive integer $k\geq 1$.

One strategy to attack this problem is to consider $L\boxtimes L^k$ on $X\times X$. Denote $\Delta$ to be the diagonal. The multiplication map above can be identified with homomorphism on global sections induced by the restriction to $\Delta$:
\[ H^0(X\times X, L\boxtimes L^k) \to H^0(X\times X, L\boxtimes L^k|_\Delta). \]
By considering the long exact sequence associated to the ideal sheaf sequence, the vanishing of $H^1(X\times X, (L\boxtimes L^k) \otimes I_{\Delta})$ would imply the surjectivity of the multiplication map.

Then the idea is to transform the vanishing from $X\times X$ to the blow up of $X\times X$ along the diagonal $\Delta$. Consider the blow up map $\pi:\text{Bl}_{\Delta}(X\times X) \to X\times X$ and denote 
\[   \mathcal{F}_k = \pi^\ast(L\boxtimes L^k)\otimes \mathcal{O}(-E),  \]
where $E$ is the exceptional divisor. We employ the lower term sequence of the Leray spectral sequence
\[ E^{pq}_2=H^p(X\times X,R^q\pi_\ast \mathcal{F}_k)\rightarrow H^{p+q}(\text{Bl}_{\Delta}(X\times X),\mathcal{F}_k).\]
That is the following exact sequence
\[ 0 \to H^1(X\times X,\pi_{\ast}\mathcal{F}_k)\to H^1(\text{Bl}_{\Delta}(X\times X),\mathcal{F}_k)\to H^0(X\times X,R^1\pi_\ast \mathcal{F}_k).\]
By the projection formula, 
\[ \pi_\ast \mathcal{F}_k=(L\boxtimes L^k)\otimes \pi_\ast \mathcal{O}(-E)=(L\boxtimes L^k)\otimes I_{\Delta}.\]
Hence to establish the vanishing of $H^1(X\times X,(L\boxtimes L^k)\otimes I_{\Delta})$, it suffices to show that $H^1(\text{Bl}_{\Delta}(X\times X),\mathcal{F}_k)=0$.

We know that $\text{Bl}_{\Delta}(X\times X)$ is canonically isomorphic to $X^{[2,1]}$. Let the first and second projection map from $X\times X$ to $X$ to be $p_1$ and $p_2$. We choose an identification such that $\pi \circ p_1=\mathrm{pr}_b$ and $\pi \circ p_2=\text{res}$.

We will compute the class of $\mathcal{F}_k-K_{X^{[2,1]}}$ in $N^1(X^{[2,1]})$:
\begin{lem}\label{lem:numerical class}
On $X^{[2,1]}$, we have
\[ \mathcal{F}_k-K_{X^{[2,1]}}\equiv \mathrm{pr}_b^\ast(L-K_X)+\mathrm{res}^\ast(L^k-K_X)-B^{a}\]
\end{lem}

\begin{proof}
Since the morphism $\phi:X^{[2,1]}\to X^{[1]}\times X$ can be identified with the blow up morphism, we can calculate the canonical class $K_{X^{[2,1]}}$ as in the Example \ref{ex:canonical divisor}, i.e. 
\[ K_{X^{[2,1]}}=\mathrm{pr}_b^\ast K_{X^{[1]}}+\text{res}^\ast K_X+E. \]
Then the class of $E$ is be computed by the following exact sequence \cite[Lemma 3.7]{Lehn}:
\[ 0 \to \mathcal{O}(-E) \to \mathrm{pr}_a^\ast(\mathcal{O}_X^{[2]}) \to \mathrm{pr}_b^\ast(\mathcal{O}_X^{[1]})\to 0.\]
In our language, 
\[ E=\mathcal{O}_D^b-(\mathcal{O}_D^a-\frac{1}{2}B^a)=\frac{1}{2}B^a. \]

Now we have
\begin{eqnarray*}
\mathcal{F}_k-K_{X^{[2,1]}}&=&\pi^\ast(L\boxtimes L^k)-E-K_{X^{[2,1]}}\\
&=& \mathrm{pr}_b^\ast(L)+\text{res}^\ast L^k-E-K_{X^{[2,1]}}\\
&=& \mathrm{pr}_b^\ast(L)+\text{res}^\ast L^k-E-\left(\mathrm{pr}_b^\ast K_{X^{[1]}}+\text{res}^\ast K_X+E\right)\\
&=& \mathrm{pr}_b^\ast(L^k-K_X)+\text{res}^\ast(L-K_X)-B^{a}
\end{eqnarray*}

\end{proof}

A special case of a theorem of Butler follows from our calculation \cite[Theorem 2A]{Butler}.

\begin{prop}
Let $X$ be a Hirzebruch surface, and $A$ be an ample line bundles on $X$, then $L=K_X+nA$ is projectively normal for $n \geq 4$.
\end{prop}

\begin{proof}
Since Reider's theorem applies as $n\geq 4$, we assume from the beginning that $n\geq 4$. By Lemma \ref{lem:numerical class} and the vanishing theorem of Kawamata-Viehweg, $L$ is projectively normal if the class of $\mathcal{F}_k-K_{X^{[2,1]}}$ is nef and big for $k\geq 1$. Since $L$ is ample when $n\geq 4$, $\mathrm{pr}_b^\ast L$ is nef and big. Also we notice that $\mathcal{F}_{k+1}=\mathcal{F}_k+\mathrm{pr}_b^{\ast}L$, it suffices to show $\mathcal{F}_1-K_{X^{[2,1]}}$ is nef and big. Write $A=aH+bF$, where $a$ and $b$ are positive integers since $A$ is ample. Then

\begin{eqnarray*}
\mathcal{F}_k-K_{X^{[2,1]}}&=&\mathrm{pr}_b^\ast(L-K_X)+\text{res}^\ast(L-K_X)-B^{a}\\
&=&\mathrm{pr}_b^\ast(pA)+\text{res}^\ast(pA)-B^{a}\\
&=& (na)H^b+(nb)F^b+(na)H^{\mathrm{diff}}+(nb)H^{\mathrm{diff}}-2(H+F)^b-2(H+F)^{\mathrm{diff}}+2D^a_{1,1}\\
&=& (na-2)H^b+(nb-2)F^b+(na-2)H^{\mathrm{diff}}+(nb-2)F^{\mathrm{diff}}+2D^a_{1,1}
\end{eqnarray*}
Because we have
\[\frac{1}{2}B^a=(H+F)^b+(H+F)^{\mathrm{diff}}-D^a_{1,1}. \]
Then by the Proposition \ref{prop:universal family of Hirzebruch surface}, if $n\geq 4, \mathcal{F}_1-K_{X^{[1,2]}}$ is nef and big. Because it lies in the interior of the nef cone, which is the ample cone by the theorem of Kleiman. Hence that is also true for all $k\geq 1$. 
\end{proof}

\begin{remark}
It's natural to try to prove the vanishing of higher syzygies using the nef cones of the nested Hilbert schemes. By Voisin's Hilbert-schematic characterization of Koszul cohomology, one can transform the desired vanishing into the vanishing of the first cohomology of a line bundle on the nested Hilbert scheme. But unfortunately, that line bundle doesn't have enough positivity to grant the vanishing if we just apply Kawamata-Viehweg. It's plausible that one can do a more detailed study of positivity of these divisors to get the desired vanishing, which may be carried out in our future research.
\end{remark}
\section{Open Problems}
\label{sec: problems}

The work in this paper has only scratched the surface of studying the birational geometry of nested Hilbert schemes. 
We conclude with a myriad of related open problems.

\subsection{Nef cones}
It's natural to try to describe the nef cones for other classes of surfaces.
Classically, $k$-very ample line bundles were used to construct extremal nef divisors on $X^{[n]}$. 
This method completely computes the nef cone when $X$ is the projective plane or a Hirzebruch surface. 
However, this approach is insufficient in general.
Recently, Bridgeland stability has been used to compute the nef cone of $X^{[n]}$ on many surfaces. 
It's natural to wonder whether these ideas can be used to compute the nef cones of nested Hilbert schemes.
\begin{problem}
Let $X$ be one of the following: Abelian surface, Enriques surface, or del Pezzo surface. Find the nef cone of $X^{[n+1,n]}$ and $X^{[n,1]}$.
\end{problem}
Note, most of the del Pezzo case may follow from similar methods to this paper as they do in the Hilbert scheme case.

\subsection{Pseudoeffective cones}
Similar methods to our nef cone computation work to compute the effective cone and entire stable base locus decomposition of our two families on rational surfaces for low numbers of points. 
\begin{problem}
Determine $\overline{\text{Eff}}(X^{[n+1,n]})$ and its stable base locus decomposition for all $n$, when $X$ is $\mathbf{P}^2$ or $\mathbf{P}^1\times \mathbf{P}^1$.
\end{problem}
\cite[Theorem 1.4]{Huizenga} computed $\overline{\text{Eff}}(\left(\mathbf{P}^{2}\right)^{[n]})$ for all $n$, and \cite{Ryan} provided an approach for the case of $\mathbf{P}^1\times \mathbf{P}^1$. 
We would expect the answer to depend on the arithmetic properties of $n$ and exceptional bundles on $\mathbf{P}^2$ and $\mathbf{P}^1\times \mathbf{P}^1$.

Since $X^{[n+1,n]}$ is the blow up of $X^{[n]}\times X$ along $Z^{[n]}$, we are able to write down the class of its canonical bundle. 
In particular, for $X=\mathbf{P}^2$,
\begin{align*}
    K_{X^{[n+1,n]}} &=-3H^b-3H^{\text{diff}}+\frac{1}{2}B^{\text{diff}} \\
    &=-3H^b-3H^{\text{diff}}+(nH^a-D^a)-((n-1)H^b-D^b) \\
    &=-2H^b+(n-3)H^{\text{diff}}-D^a+D^b
\end{align*}
so when $n$ gets large, $X^{[n+1,n]}$ cease to be a log Fano variety. 
\begin{problem}
For $n>>0$, determine whether $\left(\mathbf{P}^{2}\right)^{[n+1,n]}$ is a Mori dream space.
\end{problem}

\subsection{General nested Hilbert schemes}
Let $\mathbf{n}=(n_1,n_2,\ldots,n_k)$ be a sequence of decreasing postive integers, and let $X^{[\mathbf{n}]}$ be the corresponding nested Hilbert schemes of $k$ collections of points of length $n_i$, respectively. 
A natural question would be:

\begin{problem}
Calculate Picard group and nef cone of $X^{[\mathbf{n}]}$ when $X$ is a (rational) surface.
\end{problem}

Note we expect the Picard number to be $k(\rho(X) +1)$ if $\mathbf{n}$ does not include one and one less than that if it does.

\subsubsection{Singularities of nested Hilbert schemes}
It is natural to expect that as the length of $\mathbf{n}$ increases, $X^{[\mathbf{n}]}$ becomes more and more singular. 
So then a natural question is can we classify when $X^{[\mathbf{n}]}$ is normal or $\mathbf{Q}$-factorial.

Let $X$ be a smooth projective surface.
By \cite{Fogarty2} and \cite{Cheah}, 
$X^{[n]}$ and $X^{[n+1,n]}$ are smooth but no other nested Hilbert scheme $X^{[\mathbf{n}]}$ is.
We could ask just how bad these singularities are.
\begin{problem}
For which $k$ and $n$ is $X^{[n+k,n+k-1,\ldots,n]}\text{ }\left(X^{[n+k,n]}\right)$ normal or $\mathbf{Q}$-factorial?
\end{problem}

Note that in the case of $X^{[n,1]}$, $X^{[n+1,n]}$ provides a natural resolution of singularity. It's natural to ask
\begin{problem}
Are there any natural resolutions of singularities for $X^{[n+k,n]}$?
\end{problem}

Using the fact that $X^{[n+1,n]}$ is canonically isomorphic to the projectivization of the dualizing sheaf of $Z^{[n]}$, \cite{Song} showed that $X^{[n,1]}$ always has rational singularities. 
This suggests:

\begin{problem}
Does $X^{[\mathbf{n}]}$ always have rational singularities?
\end{problem}

\subsubsection{Irreducibility of nested Hilbert scheme}
Another natural question would be irreducibility. 
\begin{problem}
When is $X^{[\mathbf{n}]}$ irreducible? If it is not irreducible, how many components does it have, and do they have good interpretations?
\end{problem}
Using liasion methods, \cite{GH} showed that $X^{[n+2,n]}$ is irreducible.
Using the perspective of commuting varieties of parabolic subalgebras of $\mathrm{GL}(n)$, the paper \cite{GG} implies the irreducibility of $X^{[n_1,...,n_k]}$ when $n_1 \leq 16$ and of $X^{[n+p,n]}$ when $p \leq 6$. Relatedly, \cite{BE} completely determined the cases when the punctual nested Hilbert schemes $X^{[n+k,n]}_0$ and $X^{[n+k,n+k-1,\cdots,n+1,n]}_0$ are irreducible using commuting varieties of parabolic subalgebras. The related work \cite{BB} suggests that as $\mathbf{n}$ becomes arbitrarly large with arbitrarily long gaps between $n_i$ and $n_{i+1}$, the variety will become reducible.

\subsection{Chow rings}
If $X$ is a smooth surface that carries a $\mathbf{C}^\ast$-action, then this action extends to $X^{[n]}$ and $X^{[n+1,n]}$. If there are only finitely many fixed points, then one can use the results of Bialynicki-Birula to give a description of Chow ring of $X^{[n]}$, which is equivalent to the cohomology ring for these surfaces. For example, Ellingsrud and Str{\o}mme determined the additive structure of $\left(\mathbf{P}^{2}\right)^{[n]}$ in \cite{ES1} and they also found a set of generators for the cohomology ring in \cite{ES2}. So we ask
\begin{problem}
Find the additive structure and a set of generators of the cohomology ring of $X^{[n+1,n]}$ when $X$ is the projective plane.
\end{problem}

\subsection{Higher codimension cycle}
Recently, there has been interest in understanding the positive and pseudoeffective cones for higher codimension cycles (c.f. \cite{FL}). 
So we propose
\begin{problem}
What are the nef, pseudoeffective, pliant, the basepoint free, and the universally pseudoeffective cones of higher codimension cycles for $X^{[n]}$ and $X^{[n+1,n]}$?
\end{problem}

\bibliographystyle{alpha}
\bibliography{reference1}{}

\end{document}